\documentclass[12pt]{amsart}

\usepackage{amssymb}
\usepackage{graphicx}
\usepackage[cmtip,all]{xy}

\newtheorem{theorem}{Theorem}[section]
\newtheorem{lemma}[theorem]{Lemma}

\newtheorem{cor}[theorem]{Corollary}

\theoremstyle{definition}

\theoremstyle{remark}
\newtheorem{remark}[theorem]{Remark}

\newcommand{\abs}[1]{\left\lvert#1\right\rvert}

\newcommand{\CB}{\mathcal{B}}

\newcommand{\CI}{\mathcal{I}}
\newcommand{\CJ}{\mathcal{J}}

\newcommand{\vol}{{\rm vol}}

\begin{document}

\title[Rational points and polynomial maps on average]
{The distribution of rational points and polynomial maps on an affine variety over a finite field on average}

\author{Kit-Ho Mak}
\address{School of Mathematics \\
Georgia Institute of Technology \\
686 Cherry Street \\
Atlanta, GA 30332-0160, USA}
\email{kmak6@math.gatech.edu}

\author{Alexandru Zaharescu}
\address{Department of Mathematics \\
University of Illinois at Urbana-Champaign \\
273 Altgeld Hall, MC-382 \\
1409 W. Green Street \\
Urbana, Illinois 61801, USA}
\email{zaharesc@math.uiuc.edu}

\subjclass[2010]{Primary 11G25; Secondary 11K36, 11T99}
\keywords{almost all, affine variety, rational points, polynomial maps, uniform distribution}


\begin{abstract}
Let $p$ be a prime, let $V/\mathbb{F}_p$ be an absolutely irreducible affine variety inside the affine $r$-space. In this paper, we consider the problem of how often a box $\mathcal{B}$ will contain the expected number of points. In particular, we give a lower bound on the volume of $\mathcal{B}$ that guarantees almost all translations of $\mathcal{B}$ in the $r$-space contain the expected number of points. This shows that the Weil estimate holds in smaller regions in an ``almost all'' sense.
\end{abstract}

\maketitle

\section{Introduction and statements of results}

Let $p$ be a prime, and let $V\subseteq\mathbb{A}^r_p:=\mathbb{A}^r(\mathbb{F}_p)$ be an absolutely irreducible affine variety over $\mathbb{F}_p$ of dimension $n$ and degree $d>1$, embedded in an affine $r$-space ($r\geq 2$), which is not contained in any hyperplane. We identify the affine $r$-space with the set of points with integer coordinates in the cube
$[0,p-1]^r$. For a box $\CB=\CI_1\times \ldots \CI_r \subseteq [0,p-1]^r$, we define $N_{\CB}(V)$ to be the number of $\mathbb{F}_p$-points on $V$ inside $\CB$. When $\CB=[0,p-1]^r$, we will write $N(V)$ for the number of $\mathbb{F}_p$-points on $V$. It is widely believed that the $\mathbb{F}_p$-points on $V$ are uniformly distributed in $\mathbb{A}^r_p$. That is,
\begin{equation}\label{eqnNBud}
N_{\CB}(V)\sim N(V)\cdot\frac{\vol(\CB)}{p^r}.
\end{equation}
In fact, using some standard techniques involving exponential sums, one can show that the classical Lang-Weil bound \cite{LaWe54} 
\begin{equation}\label{eqnLWfull}
N(V) = p^n+O((d-1)(d-2)p^{n-\frac{1}{2}})
\end{equation}
together with the Bombieri estimate \cite{Bom78} imply
\begin{equation}\label{eqnLW}
N_{\CB}(V) = N(V)\cdot\frac{\vol(\CB)}{p^{r}}+O(p^{n-\frac{1}{2}}\log^r{p}),
\end{equation}
and the implied constant depends only on $V$. If $f$ and $g$ are two functions of $p$, we write
\begin{equation}\label{eqnomega}
f=\Omega(g)
\end{equation}
to denote the function $f/g$ tends to infinity as $p$ tends to infinity. In other words, \eqref{eqnomega} is equivalent to $g=o(f)$.
By \eqref{eqnLWfull}, the main term of \eqref{eqnLW} dominates the error term when
\begin{equation*}
\vol(\CB)= \Omega(p^{r-\frac{1}{2}}\log^r{p}).
\end{equation*}
In those cases \eqref{eqnNBud} holds. A natural and intriguing  question that arises is whether \eqref{eqnNBud} continues to hold for smaller boxes $\CB$.

Several improvements on \eqref{eqnLW}  have been established by Shparlinski and Skorobogatov \cite{ShSk90}, Skorobogatov \cite{Sko92}, Luo \cite{Luo99} and Fouvry \cite{Fou00}. We describe here Luo's result, which will be used later. First, we fix some notations following Katz \cite{Kat99}, which were used in Luo's estimate. Assume that $\dim V\geq 2$. We first homogenize $V$ using the variable $x_0$, call the resulting projective variety $X$. Define
\begin{equation}\label{defL}
L = \{ \mathbf{x}=(x_0,\ldots,x_r)\in X | x_0=0 \},
\end{equation}
and for any nonzero $\mathbf{u}=(u_1,\ldots,u_r)$, define
\begin{equation}\label{defH}
H_{\mathbf{u}} = \{ \mathbf{x}\in X | u_1x_1+\ldots+u_rx_r=0 \}.
\end{equation}
Suppose that $X\cap L\cap H_{\mathbf{u}}$ has dimension $n-2$. Denote by $\delta_{\mathbf{u}}$ the dimension of its singular locus, i.e.
\begin{equation*}
\delta_{\mathbf{u}} = \dim(\text{Sing}(X\cap L\cap H_{\mathbf{u}})).
\end{equation*}
Here we adopt the convention that the empty variety has dimension $-1$. If $X\cap L\cap H_{\mathbf{u}}$ has dimension $n-2$ for all $\mathbf{u}$ (this is so if $X\cap L$ is not contained in any hyperplane other than $L$), we define
\begin{equation}\label{defdelta}
\delta=\delta(V):=\max_{\mathbf{u}\neq 0}\delta_{\mathbf{u}}.
\end{equation}
It is clear that $\delta\leq n-2$ for all $V$. Note that we always have $\delta=-1$ if $V$ is a nonsingular curve.

With the above notations, Luo's result states that
\begin{equation}\label{eqnLuo}
N_{\CB}(V) = N(V)\cdot\frac{\vol(\CB)}{p^{r}}+O(p^{\frac{n+1+\delta}{2}}\log^r{p}).
\end{equation}
This implies that if
\begin{equation}\label{eqnvolBlarge}
\vol(\CB) = \Omega(p^{r-\frac{n-1-\delta}{2}}\log^r{p}),
\end{equation}
then the box $\CB$ contains the expected number of points on $V$.

For smaller boxes $\CB$, we do not know if the uniform distribution \eqref{eqnNBud} continues to hold. For some special varieties, there are non-trivial upper bounds for $N_{\CB}(V)$ that hold for smaller $\CB$ when $\CB$ is a cube (i.e. when $\abs{\CI_1}=\ldots=\abs{\CI_r}$). These include the cases of the modular hyperbola in dimension one \cite{ChSh10}, modular hyperbolas in dimensions two and three \cite{CiGa11}, exponential curves \cite{ChSh10}, quadratic forms \cite{Zum11}, the graph of a polynomial \cite{CGOS12}, and hyperelliptic curves \cite{CCGHSZ11}. These bounds do not imply \eqref{eqnNBud}. Indeed, there are no known non-trivial bounds available for small $\CB$ that do not satisfy \eqref{eqnvolBlarge}.

While it seems very difficult to improve on the above bounds further to obtain non-trivial information for smaller boxes $\CB$ that do not satisfy \eqref{eqnvolBlarge}, one can expect stronger results on average. For instance, Chan \cite{Cha11} considered the number of points on average on the modular hyperbola modulo an odd number $q$, i.e.
\begin{equation*}
xy\equiv c\pmod{q},
\end{equation*}
and showed that almost all boxes satisfying
\begin{equation*}
\vol(\CB)\gg O(q^{\frac{1}{2}+\varepsilon})
\end{equation*}
have the expected number of points. He also proved similar results for the higher dimensional modular hyperbola
\begin{equation*}
x_1\ldots x_r\equiv c \pmod{q}.
\end{equation*}
Another result of this sort for the modular hyperbola with only one moving side was obtained by Gonek, Krishnaswami and Sondhi \cite{GKS02}.

In the present paper, we consider the problem of how often \eqref{eqnNBud} holds for smaller $\CB$. More precisely, we fix a box $\CB=\CI_1\times\ldots\times\CI_r$, where $\CI_j$ are intervals in $[0,p)$ with length $L_j:=\abs{\CI_j}$, and $\vol(\CB)=L_1\ldots L_r$. For a vector $\mathbf{x}=(x_1,\ldots,x_r)\in\mathbb{F}_p^r$, let $\CB_{\mathbf{x}}=\mathbf{x}+\CB$ be the translate of $\CB$ by $\mathbf{x}$. We are interested to find the proportion of boxes
$\CB_{\mathbf{x}}$ that satisfy \eqref{eqnNBud}. Recall that \eqref{eqnNBud} holds when $\CB$ satisfies \eqref{eqnvolBlarge} by the Lang-Weil bound \eqref{eqnLW}. On the other hand, if
\begin{equation}\label{eqnvolBsmall}
\vol(\CB)=O(p^{r-n}),
\end{equation}
then \eqref{eqnNBud} becomes
\begin{equation*}
N_{\CB}(V)\sim N(V)\cdot\frac{\vol(\CB)}{p^r} \sim p^n\cdot\frac{\vol(\CB)}{p^r} = o(1).
\end{equation*}
That is to say, $N_{\CB}(V)$ should be zero. By a counting argument, one can show that this is indeed true for almost all translates $\CB_{\mathbf{x}}$. Here, by ``almost all'' we mean the statement holds true with probability one as $p$ tends to infinity.

\begin{theorem}\label{thm0}
Let $p$ be a prime, and let $V$ be an absolutely irreducible affine variety in $\mathbb{A}^r_p$ of dimension $n$. Let $\CB$ be a box whose volume satisfies \eqref{eqnvolBsmall}. Then  for all $\mathbf{x}\in\mathbb{F}_p^r$ with at most $O_r(p^n\vol(\CB))$
possible exceptions, we have $N_{\CB_{\mathbf{x}}}(V)=0$.
\end{theorem}

When $\vol(\CB)$ does not satisfy any of \eqref{eqnvolBsmall} or \eqref{eqnvolBlarge}, there are no general bounds available for $N_{\CB}(V)$, and the counting argument from Theorem \ref{thm0} does not work. Nevertheless, we are able to decrease the bound \eqref{eqnvolBlarge}, by showing that \eqref{eqnNBud} holds for almost all $\CB_{\mathbf{x}}$ when $\vol(\CB)$ becomes smaller. For some $V$, we can even show that unless
\begin{equation*}
\vol(\CB) \sim p^{r-n},
\end{equation*}
one always has that \eqref{eqnNBud} holds for almost all $\CB_{\mathbf{x}}$. The key 
result here is the following estimate for the second moment.

\begin{theorem}\label{thm1}
Let $p$ be a prime, and let $V$ be an absolutely irreducible affine variety in $\mathbb{A}^r_p$ of dimension $n$ and degree $d$. Let $\delta=\delta(V)$ be defined as in \eqref{defdelta}. Let $\CB$ be a box, then
\begin{equation*}
\sum_{\mathbf{x}\in\mathbf{F}_p^r} \abs{N_{\CB_{\mathbf{x}}}(V)-N(V)\cdot\frac{\vol(\CB)}{p^r}}^2 \ll_{r,n,d} p^{n+1+\delta}\vol(\CB).
\end{equation*}
\end{theorem}

An immediate corollary of Theorem \ref{thm1} is a lower bound on $\vol(\CB)$ which guarantees that almost all $\CB_{\mathbf{x}}$ contain the expected number of points.
\begin{cor}\label{cor1}
Notations are as in Theorem \ref{thm1}. If the box $\CB$ satisfies
\begin{equation}\label{eqnvolBok}
\vol(\CB) = \Omega(p^{r-(n-1-\delta)}),
\end{equation}
then for all but $o(p^r)$ vectors $\mathbf{x}\in\mathbb{F}_p^r$, the box $\CB_{\mathbf{x}}$ satisfies
\begin{equation*}
N_{\CB_{\mathbf{x}}}(V)\sim N(V)\cdot\frac{\vol(\CB_{\mathbf{x}})}{p^r}.
\end{equation*}
\end{cor}
In particular, if $\vol(\CB)$ satisfies \eqref{eqnvolBok}, then almost all boxes $\CB_{\mathbf{x}}$ contain a point of $V$. 

We may also recover a lower bound of $\vol(\CB)$ that guarantees $N_{\CB}(V)>0$. This allows us to remove the log factor in \eqref{eqnvolBlarge}. This corollary can also be obtained as a direct consequence of the recent work by Fourvy, Kowalski, and Michel \cite{FKM13} about the ``sliding-sum method''. 
\begin{cor}\label{cor1a}
Notations are as in Theorem \ref{thm1}. For any box $\CB$ that satisfies
\begin{equation}\label{eqnvolBlarge'}
\vol(\CB) = \Omega(p^{r-\frac{n-1-\delta}{2}}),
\end{equation}
we have $N_{\CB}(V)>0$.
\end{cor}

A series of remarks are in order.

\begin{remark}
Combining Theorem \ref{thm0} and Theorem \ref{thm1}, we see that if $\delta=-1$, then \eqref{eqnNBud} holds for almost all boxes $\CB_{\mathbf{x}}$, unless $\vol(\CB)\sim p^{n-r}$. This is true in particular when $V$ is a curve, or when $V$ is a smooth complete intersection (see \cite{Kat99}).
\end{remark}

\begin{remark}
When $V$ is the variety given by $x_1\ldots x_r=c$ over $\mathbb{F}_p$, we recover the result of Chan \cite{Cha11}, apart from the constants.
\end{remark}

\begin{remark}
When $\delta=-1$, Theorem \ref{thm1} is in some sense best possible, as the following example illustrates that one can get different probabilities for different varieties in the marginal case. Let $C\subseteq \mathbb{A}^2_p$ be the curve $y^{\ell}=f(x)$, where $f(x)$ is an $\ell$-th power free polynomial in $\mathbb{F}_p[x]$. Let $\CB$ be the box $[0,1)\times[0,p)$, then $\vol(\CB)=p$. It is easy to see that as $p$ tends to infinity, the probability of having
$\CB_{\mathbf{x}}\neq\emptyset$ is exactly $1/\ell$. Therefore, this probability depends on the degree of $C$. Such a phenomenon continues to exist for cyclic covers of $\mathbb{A}^n_p$ for any dimension $n$. We leave open the problem of describing, for a general variety, the possible connection
between the probability of having $\CB_{\mathbf{x}}\neq\emptyset$ and the degree of the variety.

Moreover, for a class of curves $C$ that includes the rational curves and the hyperelliptic curves, and for $\CB=(x,x+H]\times(\alpha p,\beta p]$ with $0\leq\alpha<\beta\leq 1$, the first author shows in \cite{Mak13} that the estimation in Theorem \ref{thm1} has the correct main term.
\end{remark}

\begin{remark}
In arithmetical terms, Theorem \ref{thm1} says that for any system of congruence equations
\begin{align*}
f_1(x_1,\ldots,x_r)&\equiv 0 \pmod{p}, \\
\vdots & \\
f_m(x_1,\ldots,x_r)&\equiv 0 \pmod{p},
\end{align*}
where $f_i\in\mathbb{Z}[x_1,\ldots,x_r]$ are polynomials of degree at least $2$, one can expect a solution with probability $1$ in any box $\CB$ of size satisfying \eqref{eqnvolBok}, as long as the above system defines a non-planar absolutely irreducible affine variety in the affine $r$-space over $\mathbb{F}_p$.
\end{remark}

An interesting question that arises is whether one can in principle provide applications of
Theorem 2  which are not obtainable directly by Weil's estimates. 
For example, applications where a given problem reduces to
study the average number of points on a variety inside a box as one varies the box,
and where averaging over the moving box, as done in the present paper, leads
to better results than examining each box individually.
We mention in this connection that in the process of studying  the distribution of 
fractional parts of $n^2\alpha$,
Rudnick, Sarnak and one of the authors
\cite{RSZ01} , \cite{Zah03} have been led to
consider the number of points on a family of curves defined modulo $p$ 
which lie inside certain boxes.
In those papers each curve and box were examined individually (via Weil's estimates),
because they vary in ranges that are too short for us to be able to take advantage
of this extra average. Nevertheless, this example, which initially came from a
problem in mathematical physics (see Rudnick and Sarnak \cite{RuSa98} and the 
references therein), shows that problems from unrelated fields
may sometime lead unexpectedly to questions of the type discussed in the present paper. From this perspective, the above results, and the more general results below, may prove useful in various contexts.

After showing that \eqref{eqnNBud} holds for smaller $\CB$ in an almost all sense, 
our next aim is to understand the action of a polynomial map on these rational points. Let $\mathbf{g}=(g_1,\ldots,g_s)$ be a polynomial map from $V$ to some $\mathbb{A}_p^s$, and let $\CB'$ be a box inside the target space $\mathbb{A}_p^s$. We may ask the more general question of how many points in $V\cap\CB$ are mapped to $\CB'$ via $\mathbf{g}$, and
further if one can improve the estimate on average. When $V$ is a curve, this problem has been studied by Vajaitu and one of the author \cite{VaZa02}, Granville, Shparlinski and one of the 
authors \cite{GSZ05}, and the authors \cite{MaZa12}. These types of questions have various
applications, for example to the residue race problem \cite{GSS06,GSZ05} and to the study of distances between an element $n$ and its multiplicative inverse $n'$ modulo $p$ \cite{FKSY05}.

As a first step in our approach to the general problem mentioned above, 
we generalize results of \cite{GSZ05, MaZa12, VaZa02} to the case of an affine variety $V$. Let $N_{\CB,\CB'}(V,\mathbf{g})$ be the number of points in $V\cap\CB$ that are mapped to $\CB'$ under $\mathbf{g}$. Similar to the case of curves, one may expect 
such points to be uniformly distributed in $\mathbb{A}_p^s$, i.e.
\begin{equation}\label{eqnNB'ud}
N_{\CB,\CB'}(V,\mathbf{g}) \sim N(V)\cdot\frac{\vol(\CB)}{p^r}\cdot\frac{\vol(\CB')}{p^s}.
\end{equation}
We show that under the assumption that the volumes of 
$\CB,\CB'$ are not too small, this is the case for \textit{all} such boxes, under some mild conditions on $\mathbf{g}$. Recall that a set of polynomial functions $S=\{f_1,\ldots,f_s\}$ is \textit{linearly independent on $V$} if there are no non-trivial linear combinations
\begin{equation*}
f:=c_1f_1+\ldots+c_sf_s, \qquad c_i\in\mathbb{F}_p,
\end{equation*}
that are identically zero on $V$.
\begin{theorem}\label{thm2}
Let $p$ be a prime, let $V$ be an absolutely irreducible affine variety in $\mathbb{A}^r_p$ of dimension $n$, defined by some equations in the $r$ variables $x_1,\ldots,x_r$, and let $\mathbf{g}=(g_1,\ldots,g_s)$ be a polynomial map from $V$ to $\mathbb{A}_p^s$. Let $\CB\subseteq\mathbb{A}_p^r$, $\CB'\subseteq\mathbb{A}_p^s$ be boxes, and write $\mathbf{p}=(x_1,\ldots,x_r)$. Let $\mathbf{u}=(u_1,\ldots,u_r)\in\mathbb{F}_p^r$ and $\mathbf{v}=(v_1,\ldots,v_s)\in\mathbb{F}_p^s$ be vectors, let \begin{equation*}
F(\mathbf{p};\mathbf{u},\mathbf{v})=u_1x_1+\ldots+u_rx_r+v_1g(\mathbf{p})+\ldots+v_sg(\mathbf{p}) 
\end{equation*}
and $\tilde{F}$ be its homogenization. Define
\begin{equation*}
H_{\mathbf{u},\mathbf{v}}=\{ \mathbf{x}: \tilde{F}(\mathbf{x})=0 \}
\end{equation*}
and
\begin{equation*}
\delta=\max_{(\mathbf{u},\mathbf{v})\neq(0,0)} \dim(\textup{Sing}(X\cap L\cap H_{\mathbf{u},\mathbf{v}})).
\end{equation*}
If the set $\{1,x_1,\ldots,x_r,g_1(\mathbf{p}),\ldots,g_s(\mathbf{p})\}$ is linearly independent on $V$, then
\begin{equation*}
N_{\CB,\CB'}(V,\mathbf{g}) = N(V)\cdot\frac{\vol(\CB)}{p^r}\cdot\frac{\vol(\CB')}{p^s}+O(p^{\frac{n+1+\delta}{2}}\log^{r+s}p).
\end{equation*}
\end{theorem}
In particular, this theorem implies that every pair of boxes $\CB, \CB'$ with
\begin{equation}\label{eqnvolB'large}
\vol(\CB)\vol(\CB')=\Omega(p^{r+s-\frac{n-1-\delta}{2}}\log^{r+s}{p})
\end{equation}
contains the expected number of points on $V\cap\CB$ that are mapped into $\CB'$.

\begin{remark}
In Theorem \ref{thm2}, the assumption of linear independence of the set
\[\{1,x_1,\ldots,x_r,g_1(\mathbf{p}),\ldots,g_s(\mathbf{p})\}\]
is necessary. As an example, let $C\subseteq \mathbb{A}^2_p$ be the elliptic curve $y^{2}=x^3+x$, and let $\mathbf{g}(x,y)=(x^3+x,y^2)$, then the image lies inside the diagonal of $\mathbb{A}_p^2$. So if we take $\CB=[0,p-1]^2$ and $\CB'=[0,(p-1)/2-1]\times[(p-1)/2+1,p-1]$, then even though $\vol(\CB)\vol(\CB')$ is of order $p^4/4$, we have $N_{\CB,\CB'}(C,\mathbf{g})=0$ since $\CB'$ does not intersect the diagonal.

As an example which shows why the linear independence of the $x_i$ is needed, 
let $\mathbf{h}(x,y)=x$, and let $\CB=[0,(p-1)/2]\times[0,p-1]$, $\CB'=[(p-1)/2+1,p-1]$. Then $N_{\CB,\CB'}(C,\mathbf{h})=0$ since the $x$-coordinates of $\CB$ and $\CB'$ do not overlap.
\end{remark}

Next, we allow the boxes $\CB$ and $\CB'$ to move around the domain and 
range respectively, and prove that \eqref{eqnNB'ud} holds for smaller boxes 
in an almost all sense.
If $\vol(\CB)\vol(\CB')$ is too small, then $N_{\CB,\CB'}(V,\mathbf{g})=0$ for almost all boxes $\CB$, $\CB'$, and hence the expected number of points \eqref{eqnNB'ud} is correct with probability one.
\begin{theorem}\label{thm3}
Let $p$ be a prime, let $V$ be an absolutely irreducible affine variety in $\mathbb{A}^r_p$ of dimension $n$, and let $\mathbf{g}:V\rightarrow\mathbb{A}_p^s$ be a polynomial function on $V$. Let $\CB\subseteq\mathbb{A}^r_p$, $\CB'\subseteq\mathbb{A}^s_p$ be boxes whose volumes satisfy
\begin{equation}\label{eqnvolB'small}
\vol(\CB)\vol(\CB') = O(p^{r+s-n}),
\end{equation}
then for all $(\mathbf{x},\mathbf{y})\in\mathbb{F}_p^r\times\mathbb{F}_p^s$ with at most $O_{r,s}(p^n\vol(\CB)\vol(\CB'))$ possible exceptions, we have $N_{\CB_{\mathbf{x}},\CB'_{\mathbf{y}}}(V,\mathbf{g})=0$.
\end{theorem}

\begin{remark}
Note that one does not need to assume linear independence of the coordinate functions of $\mathbf{g}$ in Theorem \ref{thm3}.
\end{remark}

As with $N_{\CB}(V)$, we decrease the bound on $\vol(\CB)\vol(\CB')$ which guarantees that \eqref{eqnNB'ud} holds for almost all $\CB_{\mathbf{x}}$ and $\CB'_{\mathbf{y}}$.
\begin{theorem}\label{thm4}
Under the same notations and assumptions of Theorem \ref{thm2}, we have
\begin{equation*}
\sum_{\mathbf{x}\in\mathbf{F}_p^r}\sum_{\mathbf{y}\in\mathbf{F}_p^s} \abs{N_{\CB_{\mathbf{x}},\CB'_{\mathbf{y}}}(V)-N(V)\cdot\frac{\vol(\CB)}{p^r}\cdot\frac{\vol(\CB')}{p^s}}^2 \ll_{r,s,n,d} p^{n+1+\delta}\vol(\CB)\vol(\CB').
\end{equation*}
\end{theorem}
\begin{cor}\label{cor4}
Notations and assumptions are as in Theorem \ref{thm2}. If
\begin{equation}\label{eqnvolB'ok}
\vol(\CB)\vol(\CB')=\Omega(p^{r+s-(n-1-\delta)}),
\end{equation}
then
\begin{equation*}
N_{\CB_{\mathbf{x}},\CB'_{\mathbf{y}}}(V,\mathbf{g}) \sim N(V)\cdot\frac{\vol(\CB)}{p^r}\cdot\frac{\vol(\CB')}{p^s}
\end{equation*}
holds for all but $o(p^{r+s})$ vectors $\mathbf{x}$, $\mathbf{y}$.
\end{cor}
As in Corollary \ref{cor1a}, we can obtain a lower bound for $\vol(\CB)\vol(\CB')$ that guarantee $N_{\CB,\CB'}(V,\mathbf{g})>0$. This allows us to remove the log factor in \eqref{eqnvolB'large}.
\begin{cor}\label{cor4a}
Notations and assumptions are as in Theorem \ref{thm2}. For all $\CB, \CB'$ with
\begin{equation}\label{eqnvolB'large'}
\vol(\CB)\vol(\CB')=\Omega(p^{r+s-\frac{n-1-\delta}{2}}),
\end{equation}
we have
\begin{equation*}
N_{\CB,\CB'}(V,\mathbf{g})>0.
\end{equation*}
\end{cor}

When $\CB'$ is the full cube
$[0,p-1]^s$ it is clear that we can remove all the assumptions for $\mathbf{g}$. In that case Theorem \ref{thm2} reduces to Luo's result \eqref{eqnLuo}, Theorem \ref{thm3} reduces to Theorem \ref{thm0}, and Theorem \ref{thm4} reduces to Theorem \ref{thm1}. We will therefore proceed directly to the proofs of Theorems \ref{thm2}, \ref{thm3} and \ref{thm4}.

\section{Lemmas on exponential sums}

In this section we recall two lemmas on exponential sums which will be useful later.
Let $e_p(x)=e^{2\pi ix/p}$. Assume $\dim V\geq 2$, recall that $X$ is the homogenization of $V$ using the variable $x_0$, and $L$, $H_{\mathbf{u}}$ are defined by \eqref{defL} and \eqref{defH} respectively. Let $\delta_{\mathbf{u}}$ denote the dimension of the singular locus of $X\cap L\cap H_{\mathbf{u}}$. One has the following estimate of exponential sums in terms of $\delta_{\mathbf{u}}$ (see  Katz \cite[Theorem 5]{Kat99}).

\begin{lemma}\label{lemKatz}
Let $V\subseteq\mathbb{A}_p^r$ be an irreducible affine variety over $\mathbb{F}_p$ of dimension $n$ and degree $d$, not contained in any hyperplane, and let $X$ be its homogenization. Let $f\in\mathbb{F}_p[x_1,\ldots,x_r]$ be a polynomial, and let $\tilde{f}$ be its homogenization. Let $L$ be as in \eqref{defL}, and define
\begin{equation}\label{defH1}
P = \{ \mathbf{x}\in X | \tilde{f}(\mathbf{x})=0 \}.
\end{equation}
Suppose that $X\cap L\cap P$ has dimension $n-2$, and let $\delta$ be the dimension of its singular locus, then
\begin{equation}\label{eqnKatz}
\abs{\sum_{\mathbf{x}\in V}e_p(f(\mathbf{x}))}\leq (4d+9)^{n+r}p^{\frac{n+1+\delta}{2}}.
\end{equation}
\end{lemma}

\begin{remark}
For the case of curves, one can use Bombieri's result \cite{Bom78} instead of Lemma \ref{lemKatz} to obtain an estimate of the same strength, but with the assumptions (except $f$ being 
nonconstant on $V$) in Lemma \ref{lemKatz} dropped.
\end{remark}

The second lemma is the following estimate.

\begin{lemma}\label{lem2}
Let $p$ be a large prime. For any interval $\CI$, we have
\begin{equation*}
\sum_{t\neq 0\textup{~mod $p$}} \abs{\sum_{m\in\CI}e_p(tm)} \leq 2p\log{p}.
\end{equation*}
\end{lemma}
\begin{proof}
Let $\CI\cap\mathbb{Z} = \{l,l+1,\dots,l+h-1\}$, where $h=\abs{\CI}$. Then
\begin{equation*}
\sum_{m\in\CI}e_p(tm) =
\begin{cases}
h & \text{if}~ t=0, \\
\left(e^{\frac{-2\pi itl}{p}}\right)\frac{1-e^{-2\pi ith/p}}{1-e^{-2\pi it/p}} & \text{if}~ t\neq 0.
\end{cases}
\end{equation*}
Hence if $t\neq 0$,
\begin{equation*}
\abs{\sum_{m\in\CI}e_p(tm)} \leq \frac{2}{\abs{1-e^{-2\pi it/p}}}.
\end{equation*}
Note that $\abs{1-e^{-2\pi it/p}} = 2\abs{\sin{(\pi t/p)}} \geq \frac{\pi\abs{s}}{p}$ for $p$ large enough, where $s$ is the least absolute residue of $t$ modulo $p$. We thus obtain the estimate
\begin{equation*}
\abs{\sum_{m\in\CI}e_p(tm)} \leq \frac{2p}{\pi\abs{s}} \leq \frac{p}{\abs{s}}.
\end{equation*}
The lemma is obtained by summing over
$1\leq\abs{s}\leq (p-1)/2$, using 
\begin{equation*}
1+\frac{1}{2}+\dots+\frac{1}{\frac{p-1}{2}} \leq \log{p}.
\end{equation*}
\end{proof}

\section{Proof of Theorem \ref{thm2}}

Let $\CI_i, \CJ_j$ be intervals, and let $\CB=\CI_1\times\ldots\times\CI_r$, $\CB'=\CJ_1\times\ldots\times\CJ_s$ be boxes. From the orthogonality of exponential sums
\begin{equation*}
\sum_{u_i \textup{~mod $p$}}e_p(u_i(n_i-m_i)) =
\begin{cases}
p, & \text{~if~} n_i=m_i, \\
0, & \text{~otherwise,}
\end{cases}
\end{equation*}
we have
\begin{equation*}
\sum_{m_i\in \CI_i}\sum_{u_i \textup{~mod $p$}}e_p(u_i(n_i-m_i)) =
\begin{cases}
p, & \text{~if~} n_i\in \CI_i, \\
0, & \text{~otherwise.}
\end{cases}
\end{equation*}
Next, we express the quantity $N_{\CB,\CB'}(V,\mathbf{g})$ in terms of exponential sums. Let $\mathbf{u}=(u_1,\ldots,u_r)$, $\mathbf{v}=(v_1,\ldots,v_s)$, $\mathbf{m}=(m_1,\ldots,m_r)$, $\mathbf{n}=(n_1,\ldots,n_s)$ and $\mathbf{z}=(z_1,\ldots,z_r)$, then
\begin{align}
N_{\CB}(V)&=\sum_{\mathbf{z}\in V}\frac{1}{p^{r+s}} \sum_{\mathbf{m}\in\CB}\sum_{\mathbf{n}\in\CB'} \sum_{\mathbf{u}\in\mathbb{F}_p^r}\sum_{\mathbf{v}\in\mathbb{F}_p^s} \prod_{i=1}^r e_p((z_i-m_i)u_i)\prod_{j=1}^s e_p((g_i(\mathbf{z})-n_j)v_j) \nonumber \\
&= \frac{1}{p^{r+s}}\prod_{i=1}^r\left( \sum_{u_i\textup{~mod $p$}}\sum_{m_i\in\CI_i} e_p(-m_iu_i) \right) \prod_{j=1}^s\left( \sum_{v_j\textup{~mod $p$}}\sum_{n_j\in\CJ_j} e_p(-n_jv_j) \right) \nonumber \\
&\qquad\times\sum_{\mathbf{z}\in V}e_p(z_1u_1+\ldots+z_ru_r+g_1(\mathbf{z})v_1+\ldots+g_s(\mathbf{z})v_s) \nonumber \\
& = M+E, \label{eqnthm21}
\end{align}
where the main term $M$ corresponds to the sum of terms with $\mathbf{u}=\mathbf{v}=0$, 
and $E$ corresponds to the sum of all other terms. We have
\begin{equation}\label{eqnM}
M = \frac{1}{p^{r+s}}\vol(\CB)\vol(\CB')N(V),
\end{equation}
where $N(V)$ is the number of points on $V$. On the other hand, because the set $\{ 1,x_1,\ldots,x_r,g_1(\mathbf{p}),\ldots,g_s(\mathbf{p}) \}$ is linearly independent on $V$, the variety $X\cap L\cap H_{\mathbf{u},\mathbf{v}}$ has dimension $n-2$ if $(\mathbf{u},\mathbf{v})\neq(0,0)$. Thus we may apply Lemma \ref{lemKatz} and Lemma \ref{lem2} to obtain
\begin{align}
E &\ll \frac{p^{\frac{n+1+\delta}{2}}}{p^{r+s}}\prod_{i=1}^r\left( \sum_{u_i\textup{~mod $p$}}\abs{\sum_{m_i\in\CI_i} e_p(-m_iu_i)} \right) \prod_{j=1}^s\left( \sum_{v_j\textup{~mod $p$}}\abs{\sum_{n_j\in\CJ_j} e_p(-n_jv_j)} \right) \nonumber \\
& \ll \frac{p^{\frac{n+1+\delta}{2}}}{p^{r+s}}(2p\log{p}+p)^{r+s} \nonumber \\
& \ll p^{\frac{n+1+\delta}{2}}\log^{r+s}{p}. \label{eqnE}
\end{align}
Using this and \eqref{eqnM} in \eqref{eqnthm21} yields Theorem \ref{thm2}.

\section{A counting argument: Proof of Theorem \ref{thm3}}

Let $\CB=\CI_1\times\ldots\times\CI_r$ and $\CB'=\CJ_1\times\ldots\times\CJ_s$. Since we are considering all translates $\CB_{\mathbf{x}}$ and $\CB'_{\mathbf{y}}$, we may assume that $\CI_i=[0,L_i)$ and $CJ_j=[0,L_j')$, where $L_i=\abs{\CI_i}$ and $L_j=\abs{\CJ_j}$. Let $L$ and $L'$ be the lattices
\begin{align*}
L &= \{ \mathbf{x}=(x_1,\ldots,x_r)|x_i=a_iL_i, a_i\in\mathbb{Z}, 0\leq a_i\leq p/L_i \}, \\
L' &= \{ \mathbf{y}=(y_1,\ldots,y_s)|y_j=b_jL_j, b_j\in\mathbb{Z}, 0\leq b_j\leq p/L'_j \}.
\end{align*}
Then the boxes $\CB_{\mathbf{x}}$ for $\mathbf{x}\in L$ are non-overlapping, except those with $a_j=[p/L_j]$ may overlap with some other boxes with $a_j=0$. Similarly, the boxes $\CB'_{\mathbf{y}}$ are also non-overlapping except for the boundary cases. In particular, any point in $V$ can only be inside at most $2^r$ boxes $\CB_{\mathbf{x}}$, and can be mapped inside at most $2^s$ boxes $\CB'_{\mathbf{y}}$. There are $p^{r+s}/(\vol(\CB)\vol(\CB'))$ boxes in total, and the number of points in $V$ satisfies $N(V)=O(p^n)$. By condition \eqref{eqnvolBsmall}, we have
\begin{equation*}
p^n=O\left(\frac{p^{r+s}}{\vol(\CB)\vol(\CB')}\right).
\end{equation*}
Thus, regardless of the choice of $\mathbf{g}$,
for almost all pairs of boxes $\CB_{\mathbf{x}}$, $\CB'_{\mathbf{y}}$ with $\mathbf{x}\in L$, 
$\mathbf{y}\in L'$, we must have $N_{\CB_{\mathbf{x}},\CB'_{\mathbf{y}}}(V,\mathbf{g})=0$,  
except for at most $O_{r,s}(p^n)$ of them. 
Repeating the above argument to all pairs $(\mathbf{x}_0,\mathbf{y}_0)+(L\times L')$ with $\mathbf{x}_0\in\CB$ and $\mathbf{y}_0\in\CB'$ proves Theorem \ref{thm3}.

\begin{remark}
The counting argument here works more generally for a $V$ that needs not be irreducible. Since $V$ has a finite  number of components, and each component has $O(p^n)$ rational points, $V$ itself has only $O(p^n)$ points as well. Thus the argument above still carries through in that case.
\end{remark}

\section{Proof of Theorem \ref{thm4}}

Let
\begin{equation*}
E_{\CB_{\mathbf{x}},\CB'_{\mathbf{y}}} = N_{\CB_{\mathbf{x}},\CB'_{\mathbf{y}}}(V,\mathbf{g})-N(V)\cdot\frac{\vol(\CB)}{p^r}\cdot\frac{\vol(\CB')}{p^s}
\end{equation*}
be the difference between $N_{\CB_{\mathbf{x}},\CB'_{\mathbf{y}}}(V,\mathbf{g})$ and the expected number of points. In this section, we calculate the second moment of $E_{\CB_{\mathbf{x}},\CB'_{\mathbf{y}}}$.

Let $\CB=\CI_1\times\ldots\times\CI_r$ and $\CB'=\CJ_1\times\ldots\times\CJ_s$. By \eqref{eqnthm21} and \eqref{eqnM}, we see that
\begin{align}
E_{\CB,\CB'} &= \frac{1}{p^{r+s}}\sum_{(\mathbf{u},\mathbf{v})\neq (0,0)\in\mathbb{F}_p^r\times\mathbf{F}_p^s}\sum_{\substack{0\leq m_i< L_i \\ 1\leq i\leq r}} e_p(-m_1u_1-\ldots-m_ru_r) \nonumber \\
&\qquad \times\sum_{\substack{0\leq n_j< L'_j \\ 1\leq j\leq s}} e_p(-n_1v_1-\ldots-n_sv_s) \nonumber \\
&\qquad\qquad \times\sum_{\mathbf{z}\in V}e_p(z_1u_1+\ldots+z_ru_r+g_1(\mathbf{z})v_1+\ldots+g_s(\mathbf{z})v_s). \label{eqnthm41}
\end{align}

Since we are moving the boxes $\CB, \CB'$ around the whole spaces, we may assume that $\CI_i=[0,L_i)$ and $\CJ_j=[0,L_j')$. Then for any $\mathbf{x}=(x_1,\ldots,x_r)\in\mathbb{F}_p^r$, $\mathbf{y}=(y_1,\ldots,y_s)\in\mathbb{F}_p^s$, we can write
\begin{align*}
\CB_{\mathbf{x}} &= \mathbf{x}+\CB = [x_1,x_1+L_1)\times\ldots\times[x_r,x_r+L_r), \\
\CB'_{\mathbf{y}} &= \mathbf{y}+\CB' = [x_1,x_1+L_1')\times\ldots\times[x_s,x_s+L_s').
\end{align*}
Therefore, \eqref{eqnthm41} implies that
\begin{align*}
E_{\CB_{\mathbf{x}},\CB'_{\mathbf{y}}} &= \frac{1}{p^{r+s}}\sum_{(\mathbf{u},\mathbf{v})\neq (0,0)} \sum_{\substack{0\leq m_i< L_i \\ 1\leq i\leq r}} e_p(-(m_1+x_1)u_1-\ldots-(m_r+x_r)u_r) \\
& \qquad \times \sum_{\substack{0\leq n_j< L'_j \\ 1\leq j\leq s}} e_p(-(n_1+x_1)v_1-\ldots-(n_s+y_s)v_s) \\
&\qquad\qquad \times \sum_{\mathbf{z}\in V}e_p(z_1u_1+\ldots+z_ru_r+g_1(\mathbf{z})v_1+\ldots+g_s(\mathbf{z})v_s).
\end{align*}

We remark that by \eqref{eqnE}, we have
\begin{equation*}
E_{\CB_{\mathbf{x}},\CB'_{\mathbf{y}}} = O(p^{\frac{n+1+\delta}{2}}\log^{r+s}{p})
\end{equation*}
for any $\mathbf{x},\mathbf{y}$, but this is not strong enough for our purpose.

The second moment of $E_{\CB_{\mathbf{x}},\CB'_{\mathbf{y}}}$ is given by
\begin{align}
& \sum_{\mathbf{x}\in\mathbf{F}_p^r}\sum_{\mathbf{y}\in\mathbf{F}_p^s} E_{\CB_{\mathbf{x}},\CB'_{\mathbf{y}}}^2 \nonumber \\
=& \frac{1}{p^{2(r+s)}}\sum_{\mathbf{x}\in\mathbb{F}_p^r}\sum_{\mathbf{y}\in\mathbb{F}_p^s}\sum_{(\mathbf{u}_1,\mathbf{v}_1)\neq (0,0)}\sum_{(\mathbf{u}_2,\mathbf{v}_2)\neq (0,0)}\sum_{\substack{0\leq m_{1i}< L_i \\ 1\leq i\leq r}}\sum_{\substack{0\leq m_{2i}< L_i \\ 1\leq i\leq r}}\sum_{\substack{0\leq n_{1j}< L'_j \\ 1\leq j\leq s}}\sum_{\substack{0\leq n_{2j}< L'_j \\ 1\leq j\leq s}}  \nonumber \\
&~ e_p\left(-\sum_{i=1}^r(m_{1i}+x_i)u_{1i}+\sum_{i=1}^r(m_{2i}+x_i)u_{2i}\right) \nonumber \\
&~~ \times e_p\left(-\sum_{j=1}^s(n_{1j}+y_j)v_{1j}+\sum_{j=1}^s(n_{2j}+y_j)v_{2j}\right) \nonumber \\
&\qquad \times\sum_{\mathbf{z}_1\in V}\sum_{\mathbf{z}_2\in V} e_p\left(\sum_{i=1}^r z_{1i}u_{1i}+\sum_{j=1}^s g_j(\mathbf{z}_1)v_{1s}\right) \nonumber \\
&\qquad\qquad\qquad \times e_p\left(-\sum_{i=1}^r z_{2i}u_{2i}-\sum_{j=1}^s g_j(\mathbf{z}_2)v_{2s}\right), \label{eqnS1}
\end{align}
where $u_1=(u_{11},\ldots,u_{1r})$, $u_2=(u_{21},\ldots,u_{2r})$, and similarly for $\mathbf{v}_1$, $\mathbf{v}_2$, $\mathbf{z}_1$, $\mathbf{z}_2$.
To simplify notations, we temporarily write
\begin{align*}
A_1 &:= -\sum_{i=1}^r m_{1i}u_{1i}+\sum_{i=1}^r m_{2i}u_{2i}, \\
A_2 &:= -\sum_{j=1}^s n_{1j}v_{1j}+\sum_{j=1}^s n_{2j}v_{2j}, \\
B_1 &:= \sum_{i=1}^r z_{1i}u_{1i}+\sum_{j=1}^s g_j(\mathbf{z}_1)v_{1s}, \\
B_2 &:= -\sum_{i=1}^r z_{2i}u_{2i}-\sum_{j=1}^s g_j(\mathbf{z}_2)v_{2s}.
\end{align*}
Then by changing the order of summation in \eqref{eqnS1}, the sum becomes
\begin{multline*}
S:=\frac{1}{p^{2(r+s)}}\sum_{(\mathbf{u}_1,\mathbf{v}_1)\neq (0,0)}\sum_{(\mathbf{u}_2,\mathbf{v}_2)\neq (0,0)}\sum_{\substack{0\leq m_{1i}< L_i \\ 1\leq i\leq r}}\sum_{\substack{0\leq m_{2i}< L_i \\ 1\leq i\leq r}}\sum_{\substack{0\leq n_{1j}< L'_j \\ 1\leq j\leq s}}\sum_{\substack{0\leq n_{2j}< L'_j \\ 1\leq j\leq s}} \\
\sum_{\mathbf{z}_1\in V}\sum_{\mathbf{z}_2\in V} e_p(A_1)e_p(A_2) e_p(B_1)e_p(B_2) \\
\times \sum_{\mathbf{x}\in\mathbb{F}_p^r}\sum_{\mathbf{y}\in\mathbb{F}_p^s}e_p(x_1(u_{21}-u_{11}))\ldots e_p(x_r(u_{2r}-u_{1r})) \\
\times e_p(y_1(v_{21}-v_{11}))\ldots e_p(y_s(v_{2s}-v_{1s})).
\end{multline*}
The innermost sum is zero unless $\mathbf{u}_1=\mathbf{u}_2$ and $\mathbf{v}_1=\mathbf{v}_2$, and in that case it equals $p^{r+s}$. Writing $\mathbf{u}=\mathbf{u}_1=\mathbf{u}_2$ and $\mathbf{v}=\mathbf{v}_1=\mathbf{v}_2$, the above sum becomes
\begin{multline*}
\frac{1}{p^{r+s}}\sum_{(\mathbf{u},\mathbf{v})\neq (0,0)}\sum_{\substack{0\leq m_{1i}< L_i \\ 1\leq i\leq r}}\sum_{\substack{0\leq m_{2i}< L_i \\ 1\leq i\leq r}}\sum_{\substack{0\leq n_{1j}< L'_j \\ 1\leq j\leq s}}\sum_{\substack{0\leq n_{2j}< L'_j \\ 1\leq j\leq s}} e_p(A_1)e_p(A_2) \\
\times \sum_{\mathbf{z}_1\in V}e_p\left(\sum_{i=1}^r z_{1i}u_{1i}+\sum_{j=1}^s g_j(\mathbf{z}_1)v_{1s}\right) \sum_{\mathbf{z}_2\in V}e_p\left(-\sum_{i=1}^r z_{2i}u_{2i}-\sum_{j=1}^s g_j(\mathbf{z}_2)v_{2s}\right).
\end{multline*}
The sums over $\mathbf{z}_1$ and $\mathbf{z}_2$ can be estimated using Lemma \ref{lemKatz} as $(\mathbf{u},\mathbf{v})\neq (0,0)$. The above sum $S$ is thus bounded by
\begin{equation}\label{eqnthm11}
S \ll_V \frac{p^{n+1+\delta}}{p^{r+s}}\sum_{(\mathbf{u},\mathbf{v})\neq (0,0)}\left|\sum_{\substack{0\leq m_{1i}< L_i \\ 1\leq i\leq r}}\sum_{\substack{0\leq m_{2i}< L_i \\ 1\leq i\leq r}}\sum_{\substack{0\leq n_{1j}< L'_j \\ 1\leq j\leq s}}\sum_{\substack{0\leq n_{2j}< L'_j \\ 1\leq j\leq s}} e_p(A_1)e_p(A_2) \right|.
\end{equation}
We now return to the definition of $A_1$ and $A_2$. For any fixed $(\mathbf{u},\mathbf{v})\neq (0,0)$, the sum inside the above absolute value equals
\begin{equation*}
\left|\sum_{\substack{0\leq m_{i}< L_i \\ 1\leq i\leq r}}\sum_{\substack{0\leq n_{j}< L'_j \\ 1\leq j\leq s}}e_p(-m_1u_1-\ldots-m_ru_r-n_1v_1-\ldots-n_sv_s)\right|^2,
\end{equation*}
which is positive. Hence we can remove the absolute sign in \eqref{eqnthm11} and move the sum over $(\mathbf{u},\mathbf{v})$ inside. This gives
\begin{align}
S &\ll_V \frac{p^{n+1+\delta}}{p^{r+s}}\sum_{\substack{0\leq m_{1i}< L_i \\ 1\leq i\leq r}}\sum_{\substack{0\leq m_{2i}< L_i \\ 1\leq i\leq r}}\sum_{\substack{0\leq n_{1j}< L'_j \\ 1\leq j\leq s}}\sum_{\substack{0\leq n_{2j}< L'_j \\ 1\leq j\leq s}} \nonumber \\
&\qquad \sum_{(\mathbf{u},\mathbf{v})\neq (0,0)} e_p(u_1(m_{21}-m_{11}))\ldots e_p(u_r(m_{2r}-m_{1r})) \nonumber \\
&\qquad\qquad\times e_p(v_1(n_{21}-n_{11}))\ldots e_p(v_s(n_{2s}-n_{1s})) \nonumber \\
&= \frac{p^{n+1+\delta}}{p^{r+s}}\sum_{\substack{0\leq m_{1i}< L_i \\ 1\leq i\leq r}}\sum_{\substack{0\leq m_{2i}< L_i \\ 1\leq i\leq r}}\sum_{\substack{0\leq n_{1j}< L'_j \\ 1\leq j\leq s}}\sum_{\substack{0\leq n_{2j}< L'_j \\ 1\leq j\leq s}} \nonumber \\
&\qquad \sum_{\mathbf{u}\in\mathbb{F}_p^r}\sum_{\mathbf{v}\in\mathbb{F}_p^s} e_p(u_1(m_{21}-m_{11}))\ldots e_p(u_r(m_{2r}-m_{1r})) \nonumber \\
&\qquad\qquad \times e_p(v_1(n_{21}-n_{11}))\ldots e_p(v_s(n_{2s}-n_{1s})) \nonumber \\
&\qquad\qquad\qquad - \frac{p^{n+1+\delta}(\vol(\CB))^2(\vol(\CB'))^2}{p^{r+s}}. \label{eqnS2}
\end{align}
The sums over $\mathbf{u}$ and $\mathbf{v}$ vanish unless $(m_{11},\ldots,m_{1r})=(m_{21},\ldots,m_{2r})$ and $(n_{11},\ldots,n_{1s})=(n_{21},\ldots,n_{2s})$, in which case the sum 
equals $p^{r+s}$. Therefore, \eqref{eqnS2} amounts to
\begin{equation*}
S \ll p^{n+1+\delta}\vol(\CB)\vol(\CB').
\end{equation*}
This proves Theorem \ref{thm4}.

Now we proceed to prove Corollary \ref{cor4} and \ref{cor4a}. Let $N$ be the number of pairs $(\mathbf{x},\mathbf{y})$ such that \eqref{eqnNB'ud} does not hold, and let $E$ be the exception set of these $(\mathbf{x},\mathbf{y})$.
By Theorem \ref{thm4},
\begin{equation*}
\sum_{(\mathbf{x},\mathbf{y})\in E}\abs{N_{\CB_{\mathbf{x}},\CB'_{\mathbf{y}}}(V,\mathbf{g})-N(V)\cdot\frac{\vol(\CB)\vol(\CB')}{p^{r+s}}}^2 \ll_V p^{n+1+\delta}\vol(\CB)\vol(\CB').
\end{equation*}
On the other hand, since 
\eqref{eqnNB'ud} does not hold, we have
\begin{align*}
\abs{N_{\CB_{\mathbf{x}},\CB'_{\mathbf{y}}}(V,\mathbf{g})-N(V)\cdot\frac{\vol(\CB)\vol(\CB')}{p^{r+s}}}^2 &\gg N(V)^2\cdot\frac{\vol(\CB)^2\vol(\CB')^2}{p^{2(r+s)}} \\
&\sim p^{2n}\cdot\frac{\vol(\CB)^2\vol(\CB')^2}{p^{2(r+s)}}
\end{align*}
for all $(\mathbf{x},\mathbf{y})\in E$. This implies
\begin{equation*}
N \cdot p^{2n}\cdot\frac{\vol(\CB)^2\vol(\CB')^2}{p^{2(r+s)}} \ll p^{n+1+\delta}\vol(\CB)\vol(\CB'),
\end{equation*}
which reduces to
\begin{equation*}
N \ll_V \frac{p^{2r+2s-n+1+\delta}}{\vol(\CB)\vol(\CB')}.
\end{equation*}
Hence, if $\vol(\CB)\vol(\CB')$ satisfies \eqref{eqnvolB'ok}, then $N = o(p^{r+s})$. This gives Corollary \ref{cor4}.

The proof of Corollary \ref{cor4a} is similar. Write
\begin{align*}
\CB = [b_1,b_1+L_1)\times\ldots\times[b_r,b_r+L_r), \\
\CB' = [b_1',b_1'+L_1')\times\ldots\times[b_s,b_s+L_s'),
\end{align*}
so that $\vol(\CB)=L_1\cdots L_r$ and $\vol(\CB')=L_1'\cdots L_s'$. If $N_{\CB,\CB'}(V,\mathbf{g})=0$, then
\begin{equation}\label{eqncor4a}
N_{\frac{1}{2}\CB_{\mathbf{x}},\frac{1}{2}\CB'_{\mathbf{y}}}=0
\end{equation}
for all $\mathbf{x}=(x_1,\ldots,x_r)$, $\mathbf{y}=(y_1,\ldots,y_s)$ with $0\leq x_i\leq L_i/2$ and $0\leq y_j\leq L_j'/2$. Here $\frac{1}{2}\CB$ is the box obtained by halving the sides of $\CB$ and with the same lower-left corner, similarly for $\frac{1}{2}\CB'$. The total number of these boxes appearing in \eqref{eqncor4a} is $\frac{1}{2^{r+s}}\vol(\CB)\vol(\CB')$. The boxes $\frac{1}{2}\CB,\frac{1}{2}\CB'$ also satisfies the conditions in Theorem \ref{thm4}, so we can apply the theorem and get
\begin{equation*}
\frac{1}{2^{r+s}}\vol(\CB)\vol(\CB') \abs{0-N(V)\cdot\frac{\vol(\CB)\vol(\CB')}{(2p)^{r+s}}}^2 \ll_V p^{n+1+\delta}\vol(\CB)\vol(\CB'),
\end{equation*} 
which simplifies to $\vol(\CB)\vol(\CB')=O(p^{r+s-\frac{n-1-\delta}{2}})$. This completes the proof of Corollary \ref{cor4a}.


\begin{thebibliography}{10}

\bibitem{Bom78}
E.~Bombieri, \emph{On exponential sums in finite fields. {II}}, Invent. Math.
  \textbf{47} (1978), no.~1, 29--39.

\bibitem{Cha11}
T.~H. Chan, \emph{An almost all result on {$q_1q_2\equiv c\pmod{q}$}}, Monatsh.
  Math. \textbf{162} (2011), no.~1, 29--39.

\bibitem{ChSh10}
T.~H. Chan and I.~E. Shparlinski, \emph{On the concentration of points on
  modular hyperbolas and exponential curves}, Acta Arith. \textbf{142} (2010),
  no.~1, 59--66.

\bibitem{CCGHSZ11}
M.-C. Chang, J.~Cilleruelo, M.~Z. Garaev, J.~Hernandez, I.~E. Shparlinski, and
  A.~Zumalac{\'a}rregui, \emph{Concentration of points and isomorphism classes
  of hyperelliptic curves over a finite field in some thin families}, {\tt
  arXiv:1111.1543 [math.NT]}.

\bibitem{CiGa11}
J.~Cilleruelo and M.~Z. Garaev, \emph{Concentration of points on two and three
  dimensional modular hyperbolas and applications}, Geom. Funct. Anal.
  \textbf{21} (2011), no.~4, 892--904.

\bibitem{CGOS12}
J.~Cilleruelo, M.~Z. Garaev, A.~Ostafe, and I.~E. Shparlinski, \emph{On the
  concentration of points of polynomial maps and applications}, to appear in
  Math. Z.

\bibitem{FKSY05}
K.~Ford, M.~R. Khan, I.~E. Shparlinski, and C.~L. Yankov, \emph{On the maximal
  difference between an element and its inverse in residue rings}, Proc. Amer.
  Math. Soc. \textbf{133} (2005), no.~12, 3463--3468.

\bibitem{FKM13}
E.~Fourvy, E.~Kowalski, and P.~Michel, \emph{The sliding-sum method for short
  exponential sums}, {\tt arXiv:1307.0135 [math.NT]}.

\bibitem{Fou00}
E.~Fouvry, \emph{Consequences of a result of {N}. {K}atz and {G}. {L}aumon
  concerning trigonometric sums}, Israel J. Math. \textbf{120} (2000), no.~part
  A, 81--96.

\bibitem{GKS02}
S.~M. Gonek, G.~S. Krishnaswami, and V.~L. Sondhi, \emph{The distribution of
  inverses modulo a prime in short intervals}, Acta Arith. \textbf{102} (2002),
  no.~4, 315--322.

\bibitem{GSS06}
A.~Granville, D.~Shiu, and P.~Shiu, \emph{Residue races}, Ramanujan J.
  \textbf{11} (2006), no.~1, 67--94.

\bibitem{GSZ05}
A.~Granville, I.~E. Shparlinski, and A.~Zaharescu, \emph{On the distribution of
  rational functions along a curve over {$\Bbb F_p$} and residue races}, J.
  Number Theory \textbf{112} (2005), no.~2, 216--237.

\bibitem{Kat99}
N.~Katz, \emph{Estimates for ``singular'' exponential sums}, Internat. Math.
  Res. Notices (1999), no.~16, 875--899.

\bibitem{LaWe54}
S.~Lang and A.~Weil, \emph{Number of points of varieties in finite fields},
  Amer. J. Math. \textbf{76} (1954), 819--827.

\bibitem{Luo99}
W.~Luo, \emph{Rational points on complete intersections over {$\bold F_p$}},
  Internat. Math. Res. Notices (1999), no.~16, 901--907.

\bibitem{Mak13}
K.-H. Mak, \emph{The distribution of points on curves over finite fields in
  some small rectangles}, accepted by Monatsh. Math.

\bibitem{MaZa12}
K.-H. Mak and A.~Zaharescu, \emph{Poisson type phenomena for points on
  hyperelliptic curves modulo {$p$}}, Funct. Approx. Comment. Math. \textbf{47}
  (2012), no.~1, 65--78.

\bibitem{RuSa98}
Z.~Rudnick and P.~Sarnak, \emph{The pair correlation function of fractional
  parts of polynomials}, Comm. Math. Phys. \textbf{194} (1998), no.~1, 61--70.

\bibitem{RSZ01}
Z.~Rudnick, P.~Sarnak, and A.~Zaharescu, \emph{The distribution of spacings
  between the fractional parts of {$n^2\alpha$}}, Invent. Math. \textbf{145}
  (2001), no.~1, 37--57.

\bibitem{ShSk90}
I.~E. Shparlinski{\u\i} and A.~N. Skorobogatov, \emph{Exponential sums and
  rational points on complete intersections}, Mathematika \textbf{37} (1990),
  no.~2, 201--208.

\bibitem{Sko92}
A.~N. Skorobogatov, \emph{Exponential sums, the geometry of hyperplane
  sections, and some {D}iophantine problems}, Israel J. Math. \textbf{80}
  (1992), no.~3, 359--379.

\bibitem{VaZa02}
M.~Vajaitu and A.~Zaharescu, \emph{Distribution of values of rational maps on
  the {${\bf F}_p$}-points on an affine curve}, Monatsh. Math. \textbf{136}
  (2002), no.~1, 81--86.

\bibitem{Zah03}
A.~Zaharescu, \emph{Correlation of fractional parts of {$n^2\alpha$}}, Forum
  Math. \textbf{15} (2003), no.~1, 1--21.

\bibitem{Zum11}
A.~Zumalac{\'a}rregui, \emph{Concentration of points on modular quadratic
  forms}, Int. J. Number Theory \textbf{7} (2011), no.~7, 1835--1839.

\end{thebibliography}

\end{document}